\documentclass{article}

\usepackage[preprint,nonatbib]{preprint_2025}

\usepackage[utf8]{inputenc} 
\usepackage[T1]{fontenc}    
\usepackage{hyperref}       
\usepackage{url}            
\usepackage{booktabs}       
\usepackage{amsfonts}       
\usepackage{nicefrac}       
\usepackage{microtype}      
\usepackage{xcolor}         
\usepackage{amsmath}
\usepackage{cleveref}
\usepackage{graphicx}
\usepackage{multirow}
\usepackage{caption}  
\usepackage{subcaption}
\usepackage{float} 
\usepackage{enumitem}
\usepackage[linesnumbered,ruled,vlined]{algorithm2e}
\crefname{algorithm}{Algorithm}{Algorithms}
\Crefname{algorithm}{Algorithm}{Algorithms}

\newtheorem{theorem}{Theorem}[section]

\newenvironment{proof}{{\noindent\it Proof.}\quad}{\hfill $\square$\par}

\DeclareMathOperator{\diag}{diag}

\SetKwInput{KwInput}{Input}         
\SetKwInput{KwOutput}{Output} 
\title{sparseGeoHOPCA: A Geometric Solution to Sparse Higher-Order PCA Without Covariance Estimation}

\author{
    Renjie~Xu\thanks{These authors contributed equally to this work.}$^{,}$\thanks{Corresponding author.}\\
    Department of Electrical Engineering and Centre for Intelligent Multidimensional Data Analysis\\
    City University of Hong Kong\\
     Kowloon, Hong Kong SAR, P. R. of China\\
    \texttt{renjie@innocimda.com}\\
    \And
    Chong~Wu\footnotemark[1]\\
    Department of Electrical Engineering and Centre for Intelligent Multidimensional Data Analysis\\
    City University of Hong Kong\\
    Kowloon, Hong Kong SAR, P. R. of China\\
    \texttt{chong@innocimda.com} \\
    \And 
     Maolin~Che \\
    State Key Laboratory of Public Big Data and School of Mathematics and Statistics\\
    Guizhou University\\
    Guiyang, Guizhou, P. R. of China\\
    \texttt{chncml@outlook.com} \\  
    \And
  Zhuoheng~Ran\\
  Department of Electrical Engineering and Centre for Intelligent Multidimensional Data Analysis\\
  City University of Hong Kong\\
  Kowloon, Hong Kong SAR, P. R. of China\\
  \texttt{harry.ran@my.cityu.edu.hk} \\
    \And
  Yimin~Wei\\
  School of Mathematics\\
  Fudan University\\
  Yangpu, Shanghai, P. R. of China\\
  \texttt{ymwei@fudan.edu.cn} \\  
    \And
  Hong~Yan\\
  Department of Electrical Engineering and Centre for Intelligent Multidimensional Data Analysis\\
  City University of Hong Kong\\
  Kowloon, Hong Kong SAR, P. R. of China\\
  \texttt{h.yan@cityu.edu.hk} \\  
}

\begin{document}

\maketitle

\begin{abstract}
  We propose sparseGeoHOPCA, a novel framework for sparse higher-order principal component analysis (SHOPCA) that introduces a geometric perspective to high-dimensional tensor decomposition. 
  By unfolding the input tensor along each mode and reformulating the resulting subproblems as structured binary linear optimization problems, our method transforms the original nonconvex sparse objective into a tractable geometric form.
  This eliminates the need for explicit covariance estimation and iterative deflation, enabling significant gains in both computational efficiency and interpretability, particularly in high-dimensional and unbalanced data scenarios. We theoretically establish the equivalence between the geometric subproblems and the original SHOPCA formulation, and derive worst-case approximation error bounds based on classical PCA residuals, providing data-dependent performance guarantees. The proposed algorithm achieves a total computational complexity of  $O\left(\sum_{n=1}^{N} (k_n^3 + J_n k_n^2)\right)$, which scales linearly with tensor size. Extensive experiments demonstrate that sparseGeoHOPCA accurately recovers sparse supports in synthetic settings, preserves classification performance under 10× compression, and achieves high-quality image reconstruction on ImageNet, highlighting its robustness and versatility.
\end{abstract}

\section{INTRODUCTION}\label{sec: introduction} 

In this paper, we study the sparse higher-order principal component analysis (SHOPCA) problem.
Higher-order principal component analysis (HOPCA), or multilinear principal component analysis (MPCA), refers to the extension of classical principal component analysis (PCA) to tensor-structured data, enabling dimensionality reduction and pattern extraction from higher-order data structures~\cite{kolda2009tensor,lu2008mpca}. 

To address the limitations of traditional HOPCA in high-dimensional settings, the introduction of sparsity constraints has emerged as an effective strategy. 
The motivation for incorporating sparsity is fourfold: (1) Sparsity enhances interpretability by ensuring that each principal component involves only a small subset of relevant features, making the results more understandable and visually interpretable.
(2) Sparse representations promote automatic feature selection and improve compression efficiency by focusing on the most informative variables. 
(3) In high-dimensional, low-sample-size scenarios, sparsity mitigates the statistical instability of classical PCA solutions, improving estimation accuracy~\cite{johnstone2009consistency}. 
(4) SHOPCA preserves the structural integrity of the original tensor while emphasizing the most informative subset of features, making it suitable for complex real-world applications such as multimodal learning~\cite{sun2022tensorformer}, biomedical analysis~\cite{allen2012sparse}, and recommender systems~\cite{frolov2017tensor}.

However, introducing sparsity into tensor PCA results in a non-convex combinatorial optimization problem, which is generally NP-hard~\cite{choo2021complexity,hillar2013most}.
As a result, various approximate algorithms and optimization strategies have been proposed.
The seminal contributions by Lu et al.~\cite{lu2006multilinear,lu2008mpca} pioneered MPCA by decomposing tensor data via mode-wise matrix PCA, laying the foundation for multilinear analysis of high-dimensional data. 
Building on this, Allen~\cite{allen2012sparse} introduced two influential models, sparse HOSVD and sparse CP, that were among the first to incorporate sparsity-promoting penalties into the tensor PCA framework. 
These models enabled the discovery of interpretable low-rank structures while simultaneously performing feature selection. 
Subsequently, Lai et al.~\cite{lai2014multilinear} proposed the multilinear sparse PCA (MSPCA) method, which extended sparse PCA ideas to tensor-valued data and demonstrated effectiveness in applications such as face recognition.
While these methods have shown notable success in image and video analysis~\cite{liu2018improved}, brain signal processing~\cite{xu2024qr,xu2024utv}, biomedical data interpretation~\cite{zhou2016linked}, and large-scale sensor networks~\cite{rajesh2021data,zhang2021recovery}, they often fall short in handling computational efficiency across different tensor modes.
In particular, the construction and manipulation of large covariance matrices in high-dimensional settings result in significant memory and computational burdens.

In this paper, we propose a geometry-aware framework for efficiently approximating SHOPCA.
Our framework facilitates the use of existing optimization solvers in a plug-and-play manner, allowing for flexible and modular integration.
 Specifically, the framework comprises three main stages:
(i)~\textit{Tensor Preparation}: Given a SHOPCA problem, we unfold the tensor along each mode and formulate the corresponding sparse matrix PCA subproblems. Initial support sets are selected for each mode-unfolded matrix. 
(ii)~\textit{Geometric Solver for Subproblems}: For each mode-unfolded matrix, we reformulate the subproblem from a geometric perspective and solve a structured binary linear program to identify the most significant columns under sparsity constraints.
(iii)~\textit{Solution Construction}: We construct the core tensor and factor matrices by combining the solutions obtained across modes, yielding the final decomposition result.
These steps are illustrated in Figure~\ref{fig:flow}.

\begin{figure}[ht]
    \centering
    \includegraphics[width=\textwidth]{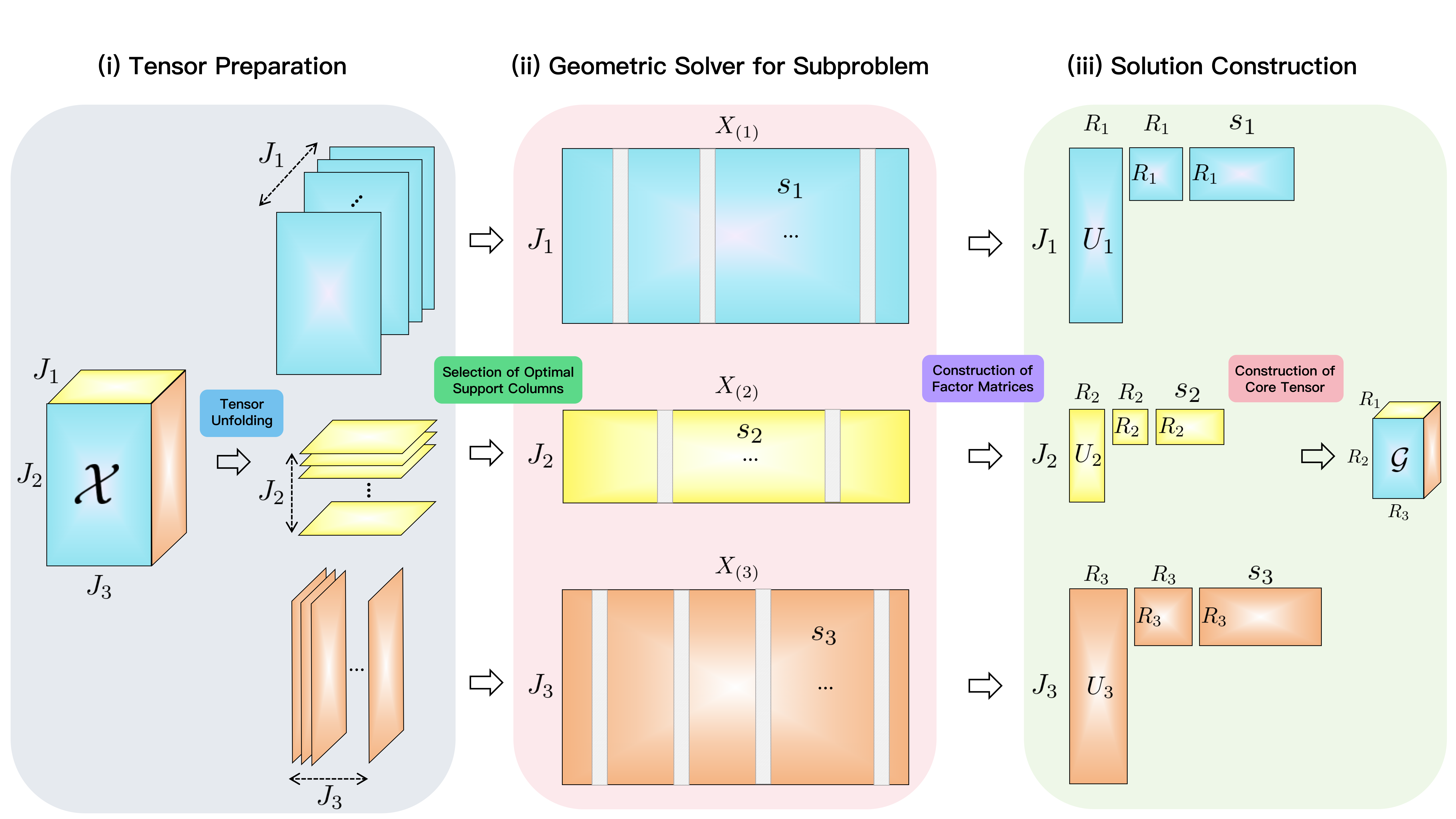}
    \caption{Illustration of the proposed \textit{sparseGeoHOPCA} pipeline on a third-order tensor. 
    (i) The tensor is unfolded along each mode to obtain matricized views. 
    (ii) A geometry-aware solver selects sparse column supports from each mode-wise unfolding. 
    (iii) Based on the selected supports, factor matrices and the core tensor are constructed to yield a sparse multilinear decomposition.}
    \label{fig:flow}
\end{figure}

Compared to prior approaches, our framework avoids explicit computation of large-scale covariance matrices, which are often infeasible in high-dimensional settings where the number of columns greatly exceeds the number of rows. 
This substantially improves computational efficiency while retaining strong approximation guarantees. 
The proposed method is more likely to find globally optimal solutions and achieve effective control of approximation error.

We theoretically justify the geometric subproblem reformulation in Section~\ref{sec: Framework of sparseGeoHOPCA}. Worst-case error bounds are derived in Section~\ref{sec: Worst-Case Upper Bound}, and computational complexity is analyzed in Section~\ref{sec: Complexity Analysis}. Extensive empirical evaluations in Section~\ref{sec: EMPIRICAL RESULTS} validate the effectiveness of the proposed method.

\textbf{Our contributions.} We summarize our contributions in the following:
\begin{itemize}[leftmargin=1em]
    \item We propose \textit{sparseGeoHOPCA}, a novel framework that introduces a geometric perspective to the SHOPCA problem.
    By formulating mode-wise column selection subproblems from a geometric viewpoint, we transform the original nonconvex sparse optimization into binary linear optimization problems. 
    This approach avoids explicit covariance matrix estimation and iterative deflation procedures, offering notable advantages in both computational efficiency and interpretability, especially in high-dimensional, unbalanced tensor settings.
    \item We theoretically establish the equivalence between the proposed geometric subproblems and the original sparse objective, and derive worst-case error bounds for SHOPCA based on residuals from classical PCA. 
    These results provide interpretable, data-dependent guarantees on the approximation quality of sparse projections.
    \item Our method significantly reduces both computational and memory overhead in high-dimensional regimes. 
    By adopting an alternating optimization strategy and solving structured support selection via geometric pruning for each tensor mode, the algorithm achieves a total computational complexity of $O\left(\sum_{n=1}^{N} (k_n^3 + J_n k_n^2)\right)$.
    This complexity scales linearly with the tensor size, ensuring strong scalability for large-scale tensor applications.
    \item We conduct extensive evaluations on support recovery, classification, and image reconstruction tasks. Experimental results demonstrate that \textit{sparseGeoHOPCA} achieves accurate support selection, maintains classification performance under 10× compression, and yields high-quality image reconstruction on ImageNet, highlighting its effectiveness, structural robustness, and generalization beyond tensor-specific applications.
\end{itemize}





\section{PROBLEM SETTING AND MOTIVATION}\label{sec: Problem Setting and Motivation}

In this section, we provide a formal definition of the sparse higher-order principal component analysis (SHOPCA) problem considered in this work, and we present the core motivation for integrating geometric feature selection techniques into a tensor-based decomposition framework.

\subsection{Notations and Definitions}\label{sec: Notations and Definitions}

Unless stated otherwise, we adopt the following notation: scalars are denoted by lowercase letters (e.g., $a, b$), vectors by bold lowercase letters (e.g., $\mathbf{v}$), matrices by uppercase letters (e.g., $M$), and tensors by calligraphic letters (e.g., $\mathcal{T}$).

Let $\mathcal{X} \in \mathbb{R}^{J_1 \times J_2 \times \cdots \times J_N}$ be an $N$-th order tensor. 
Its \textbf{mode-}$n$ \textbf{matricization}, denoted by $X_{(n)} \in \mathbb{R}^{J_n \times \prod_{i \ne n} J_i}$, rearranges the mode-$n$ fibers of $\mathcal{X}$ into columns via the unfolding operator $\mathbf{unfold}_n(\mathcal{X})$. 
The inverse operation $\mathbf{fold}_n(\cdot)$ reconstructs the tensor from its matricized form, satisfying $\mathcal{X} = \mathbf{fold}_n(X_{(n)})$.

The \textbf{mode-}$k$ product (also known as the Tucker product) of $\mathcal{X}$ with a matrix $U_k \in \mathbb{R}^{J_k \times R_k}$ is denoted by $\mathcal{Y} = \mathcal{X} \times_k U_k$, and produces a tensor of size $\mathbb{R}^{J_1 \times \cdots \times J_{k-1} \times R_k \times J_{k+1} \times \cdots \times J_N}$. 
This transformation projects the mode-$k$ fibers of $\mathcal{X}$ onto a lower-dimensional subspace, and its matrix representation is given by: $Y_{(k)} = U_k X_{(k)}$, where $Y_{(k)}$ is the mode-$k$ matricization of the resulting tensor $\mathcal{Y}$.

\subsection{Higher-Order Principal Components Analysis (HOPCA)}\label{sec: Higher-Order Principal Components Analysis (HOPCA)}

Higher-order principal components analysis (HOPCA) extends classical PCA to tensors via Tucker decomposition. 
Given $\mathcal{X} \in \mathbb{R}^{J_1 \times \cdots \times J_N}$, HOPCA approximates the tensor using:
\begin{equation}\label{equ: X-=-G}
    \mathcal{X} \approx \mathcal{G} \times_1 U_1 \times_2 U_2 \cdots \times_N U_N,
\end{equation}
where $\mathcal{G} \in \mathbb{R}^{R_1 \times \cdots \times R_N}$ is the core tensor and $U_n \in \mathbb{R}^{J_n \times R_n}$ are orthonormal factor matrices.
In component-wise form:
\begin{equation}\label{equ: X(i1,...,iN)}
\mathcal{X}(i_1,\ldots,i_N) = \sum_{\alpha_1=1}^{R_1}\cdots\sum_{\alpha_N=1}^{R_N} \mathcal{G}(\alpha_1,\ldots,\alpha_N) \prod_{n=1}^{N} U_n(i_n, \alpha_n).
\end{equation}
This representation drastically reduces the storage and computation cost from $\mathcal{O}(J_1 \cdots J_N)$ to $\mathcal{O}(R_1 \cdots R_N + \sum_{n=1}^{N} J_n R_n)$, offering significant advantages when $J_n \gg R_n$.
Optimization is typically achieved using higher-order SVD or alternating least squares (ALS), enabling effective and interpretable multilinear dimensionality reduction.

\subsection{Sparse Higher-Order Principal Components Analysis (SHOPCA)}\label{sec: Sparse Higher-Order Principal Components Analysis}

To enhance interpretability and robustness, sparse PCA has been extended to tensors via the SHOPCA framework. 
Let $\mathcal{X} \in \mathbb{R}^{J_1 \times \cdots \times J_N}$ be the data tensor and $U_n \in \mathbb{R}^{J_n \times R_n}$ the projection matrix for each mode. 
The objective is to minimize the projection error while enforcing sparsity and orthogonality:
\begin{equation}\label{equ: f(U_{1},U_{2},,U_{N})}
    f(U_1, \ldots, U_N) = \left\| \mathcal{X} - \mathcal{X} \times_1 U_1 U_1^\top \cdots \times_N U_N U_N^\top \right\|_F^2.
\end{equation}

The SHOPCA problem is thus formalized as:
\begin{equation}\label{equ: fU1U2UN}
\begin{aligned}
& \underset{U_1, \ldots, U_N}{\text{minimize}} 
& & f(U_1, \ldots, U_N) \\
& \text{subject to} 
& & \| U_n \|_0 \leq k_n, \quad U_n^\top U_n = I_{R_n}, \quad \text{for } n = 1, \ldots, N,
\end{aligned}
\end{equation}
where the $\ell_0$ constraint enforces shared row-sparsity across columns of $U_n$.

\paragraph{Motivating Idea of This Paper.}

The SHOPCA problem is inherently challenging due to its non-convex nature and lack of a closed-form solution.
Inspired by the alternating optimization strategy in HOPCA, we decompose the problem into $N$ independent subproblems (see Theorem~\ref{thm: optimization decomposition}).
Fixing all $U_m$ for $m \neq n$, the objective simplifies to the following sparse matrix approximation:
\begin{equation}\label{equ: X_n-U_nU_n^TX_n}
\begin{aligned}
& \underset{V_n}{\text{minimize}} 
& & \left\| X_{(n)} -   U_n U_n^\top X_{(n)}\right\|_F^2 \\
& \text{subject to} 
& & \|U_n\|_0 \leq k_n,\quad U_n^\top U_n = I_{R_n},
\end{aligned}
\end{equation}
where $X_{(n)}$ is the mode-$n$ unfolding of $\mathcal{X}$. In practice, when $J_n \ll \prod_{i \ne n} J_i$, traditional sparse PCA methods relying on covariance estimation become inefficient.

To overcome this bottleneck, we integrate the \emph{GeoSPCA} method into our tensor framework. 
GeoSPCA circumvents the need for covariance computation and iterative deflation by directly selecting the columns of $X_{(n)}$ with the largest Frobenius norms. 
The column selection problem is formulated as:
\begin{equation}
\label{equ:eta-constrained-problem}
\begin{aligned}
\underset{\mathbf{s} \in \{0,1\}^{\prod_{i \ne n} J_i}}{\text{maximize}} \quad & 
\sum_{j=1}^{\prod_{i \ne n} J_i} s_j \left\| X_{(n)}(:,j) \right\|_F^2 \\
\text{subject to} \quad & 
\mathbf{e}^\top \mathbf{s} \leq k_n,\quad 
\forall \sigma^n \subset \left[\prod_{i \ne n} J_i\right],\ 
\eta(\mathbf{s}^{\sigma^n}) > \eta \Rightarrow \sum_{j \in \sigma^n} s_j \leq |\sigma^n| - 1,
\end{aligned}
\end{equation}

where $\eta(\mathbf{s}) = \left\| X_{(n)}(:,\mathbf{s}) - U_n[\mathbf{s}] U_n[\mathbf{s}]^\top X_{(n)}(:,\mathbf{s}) \right\|_F^2$ measures the reconstruction error over the selected column subset, and  $U_n[\mathbf{s}]$ denotes the optimal projection basis obtained by solving $U_n[\mathbf{s}] = \arg\max_{U \in \mathbb{R}^{J_n \times R_n}} \operatorname{tr}\left(U^{\top} X_{(n)}(:,\mathbf{s}) X_{(n)}(:,\mathbf{s})^{\top} U\right) \quad \text{subject to} \quad U^{\top} U = I_{R_n}$.
Here, $\mathbf{e} \in \mathbb{R}^{\prod_{i \ne n} J_i}$ is an all-ones vector that enforces the sparsity constraint, and $\left[\prod_{i \ne n} J_i\right]$ denotes the index set $\{1, 2, \ldots, \prod_{i \ne n} J_i\}$ of all columns in $X_{(n)}$.
Each subset $\sigma^n$ represents a candidate column group that is excluded if its reconstruction error exceeds the threshold $\eta$.  
Details are further discussed in Theorem~\ref{thm: column optimaztion}.

While GeoSPCA was originally introduced for matrix-based sparse PCA, our main contribution lies in its adaptation to the multilinear tensor setting. 
This integration yields a scalable, non-iterative, and interpretable approach for sparse tensor decomposition, particularly effective in high-dimensional and unbalanced scenarios.

\section{OUR APPROACH: sparseGeoHOPCA}\label{sec: OUR APPROACH: sparseGeoHOPCA}
In this section, we propose \textit{sparseGeoHOPCA}, a geometry-aware framework for solving the SHOPCA problem.
The method combines alternating optimization with structured column selection inspired by GeoSPCA, enabling efficient and interpretable sparse tensor decomposition.

Section~\ref{sec: Algorithms} outlines the overall algorithm. 
Section~\ref{sec: Framework of sparseGeoHOPCA} details the theoretical foundation, including mode-wise decoupling and sparse PCA reformulation. 
Section~\ref{sec: Worst-Case Upper Bound} provides approximation error bounds.
Section~\ref{sec: Complexity Analysis} discusses the computational complexity.

\subsection{Algorithms}\label{sec: Algorithms}

In this section, we present the proposed algorithmic framework, \textbf{sparseGeoHOPCA} (Geometry-Aware Sparse Higher-Order PCA via Alternating Optimization), designed to solve the sparse higher-order principal component analysis (SHOPCA) problem defined in ~\eqref{equ: fU1U2UN}.

Given an input tensor $\mathcal{X} \in \mathbb{R}^{J_1 \times J_2 \times \cdots \times J_N}$, target Tucker ranks $(R_1, \ldots, R_N)$, sparsity parameters $\{k_n\}_{n=1}^N$, a tolerance threshold $\eta$, and optional initial exclusion sets for each mode, our goal is to identify sparse orthonormal factor matrices $\{U_n\}_{n=1}^N$ and a compact core tensor $\mathcal{G}$.

To this end, we employ an alternating optimization strategy that decouples the problem into $N$ independent subproblems, each corresponding to one tensor mode. 
For the $n$-th mode, we unfold the tensor into its mode-$n$ matricization $X_{(n)}$ and formulate a binary linear optimization (BLO) problem $\phi^{(n)}$ to select a support vector $\mathbf{s}^0$ that determines a subset of dominant columns.

At each iteration $t$, we compute the approximation error $\left\| X_{(n)}(:,\mathbf{s}^t) - U_{n}[\mathbf{s}^t] U_{n}[\mathbf{s}^t]^\top X_{(n)}(:,\mathbf{s}^t) \right\|_F^2$.
If this error exceeds the tolerance $\eta$, we refine the BLO problem by adding a constraint to exclude the current support $\sigma^t$, and resolve $\phi^{(n)}$ to obtain a new support $\mathbf{s}^{t+1}$. 
This process continues until the approximation error falls below the threshold.

Once all sparse factor matrices $\{U_n\}_{n=1}^N$ are obtained, the Tucker core tensor is computed as:
\begin{equation}
    \mathcal{G} = \mathcal{X} \times_1 U_1^\top \times_2 U_2^\top \cdots \times_N U_N^\top.
\end{equation}

The complete procedure of the sparseGeoHOPCA algorithm is summarized in Algorithm~\ref{alg:GeoSPCA-Tucker}.

\begin{algorithm}[htbp]
\DontPrintSemicolon
\KwInput{Tensor $\mathcal{X} \in \mathbb{R}^{J_1 \times J_2 \times \cdots \times J_N}$; target ranks $(R_1, \ldots, R_N)$; sparsity levels $\{k_n\}_{n=1}^N$; tolerance parameter $\eta$; optional initial exclusion sets.}
\KwOutput{Sparse factor matrices $\{U_n\}_{n=1}^N$ and core tensor $\mathcal{G}$ from Tucker decomposition.}

\For{$n = 1$ to $N$}{
    Matricize $\mathcal{X}$ along mode $n$ to obtain $X_{(n)} \in \mathbb{R}^{J_n \times \prod_{i \neq n} J_i}$\;
    Formulate the BLO problem $\phi^{(n)}$:
    $$
    \max_{\mathbf{s} \in \{0,1\}^{\prod_{i \neq n} J_i}} \sum_{i=1}^{\prod_{i \neq n} J_i} s_i \| X_{(n)}(:,i) \|_{F}^2, \quad \text{s.t.} \quad \mathbf{e}^\top \mathbf{s} \leq k_n
    $$
    Solve $\phi^{(n)}$ to obtain an optimal support vector $s^0$\;
    Compute $U_{n}[s^0]$ by applying PCA to $X_{(n)}(:,s^0)$\;
    
    \While{$\| X_{(n)}(:,\mathbf{s}^t) - U_{n}[\mathbf{s}^t] U_{n}[\mathbf{s}^t]^\top X_{(n)}(:,\mathbf{s}^t) \|_F^2 > \eta$}{
        Add constraint $\sum_{i \in \sigma^{n}_t} s_i \leq |\sigma_t^n| - 1$ to $\phi^{(n)}$\;
        Re-solve $\phi^{(n)}$ to obtain new support $\mathbf{s}^{t+1}$\;
        Compute $U_{n}[\mathbf{s}^{t+1}]$ via PCA on $X_{(n)}(:,\mathbf{s}^{t+1})$\;
    }
    Set $U_n \in \mathbb{R}^{J_{n}\times R_{n}}$ as the final solution for mode $n$\;
}

Compute the core tensor as:
$$
\mathcal{G} = \mathcal{X} \times_1 U_1^\top \times_2 U_2^\top \cdots \times_N U_N^\top
$$

\Return $\{U_n\}_{n=1}^N$, $\mathcal{G}$\;

\caption{sparseGeoHOPCA: Geometry-Aware Sparse Higher-Order PCA via Alternating Optimization}
\label{alg:GeoSPCA-Tucker}
\end{algorithm}

\subsection{Framework of \textit{sparseGeoHOPCA}}\label{sec: Framework of sparseGeoHOPCA}

To address the sparse higher-order principal component analysis (SHOPCA) problem defined in~\eqref{equ: fU1U2UN}, we propose a geometry-aware alternating optimization framework named \textit{sparseGeoHOPCA}. 
This method iteratively decouples the full tensor optimization problem into a sequence of sparse subproblems, each corresponding to a single tensor mode.
The decomposition is formally established in the following result.

\begin{theorem}\label{thm: optimization decomposition}
Let $(U_1, \ldots, U_{n-1}, U_{n+1}, \ldots, U_N)$ be fixed. 
Then the optimization of $U_n$ in~\eqref{equ: fU1U2UN} reduces to the sparse matrix approximation problem given in ~\eqref{equ: X_n-U_nU_n^TX_n}.
\end{theorem}

\begin{proof}
See Appendix~\ref{app: appendix a}.
\end{proof}

To solve this subproblem, we adopt a column selection strategy inspired by the \textit{GeoSPCA} method~\cite{bertsimas2022sparse}. 
The key idea is to transform the original sparse optimization into a binary linear program with geometric exclusion constraints, which can be efficiently solved using standard integer programming techniques.
The following theorem justifies this reformulation.

\begin{theorem}\label{thm: column optimaztion}
Let $s^{0}$ be an optimal solution to problem~\eqref{equ: X_n-U_nU_n^TX_n}. Then there exists a constant $\delta > 0$ such that for any $\eta \in [\eta(\mathbf{s}^{0}), \eta(\mathbf{s}^{0}) + \delta]$, any optimal solution to the selection problem in~\eqref{equ:eta-constrained-problem} is also an optimal solution to problem~\eqref{equ: X_n-U_nU_n^TX_n}.
\end{theorem}

\begin{proof}
See Appendix~\ref{app: appendix b}.
\end{proof}

\subsection{Worst-Case Upper Bound}\label{sec: Worst-Case Upper Bound}

While the alternating optimization framework of \textit{sparseGeoHOPCA} guarantees convergence to a locally optimal solution, it remains crucial to understand the quality of the solution obtained in each mode. 
In particular, we are interested in quantifying how well the sparse projection subspace captures the original data variance compared to its dense (classical PCA) counterpart.

To this end, we establish a worst-case upper bound on the reconstruction error $\eta(\mathbf{s}^0)$ obtained from solving the sparse subproblem~\eqref{equ:eta-constrained-problem}.
The bound, obtained a priori, leverages the residual matrix from classical PCA and highlights the connection between sparsity-induced loss and the spectral structure of the original data.

\begin{theorem}\label{thm: matrix error}
Let $(\mathbf{s}^{0}, U_{n}[\mathbf{s}^{0}])$ be an optimal solution to problem~\eqref{equ:eta-constrained-problem}. 
Consider the classical PCA solution $V^{*} \in \underset{V \in \mathbb{R}^{J_{n} \times R_{n}}}{\arg\min} \left\| X_{(n)} - V V^{\top} X_{(n)} \right\|_F \quad \text{subject to} \quad V^{\top}V = I_{R_n}$, and define the associated residual matrix as $\epsilon^{n} = X_{(n)} - V^{*} V^{*\top} X_{(n)}$.
Let $\sigma^{n} \subset \left[ \prod_{i \neq n} J_i \right]$ denote the indices corresponding to the $k_n$ columns of $\epsilon^{n}$ with the highest norm.
Then, the following inequality holds:
\begin{equation}\label{equ: epsilon n bound}
\eta(\mathbf{s}^{0}) \leq \left\| S^{0} \epsilon^{n} \right\|_{F}^{2} \leq \sum_{i \in \sigma^{n}} \left\| \epsilon^{n}_{i} \right\|_{F}^2,
\end{equation}
where $\epsilon^{n}_{i}$ is the $i$-th column of the residual matrix $\epsilon^{n}$, and $S^{0} = \diag (\mathbf{s}^{0})$. Moreover, the solution $\mathbf{s}^{0}$ is feasible for problem~\eqref{equ:eta-constrained-problem} when $\eta = \|S^{0} \epsilon^{n}\|_F^2$.
\end{theorem}

\begin{proof}
See Appendix~\ref{app: appendix c}.
\end{proof}

\noindent
The residual matrix $\epsilon^{n}$ can be efficiently computed using classical PCA, making the bound on $\eta(\mathbf{s}^0)$ both practical and interpretable. The result provides a data-dependent certificate of approximation quality under sparsity constraints.

We now extend this worst-case characterization to the full multilinear setting. The following result provides a global upper bound on the total approximation error of the \textit{sparseGeoHOPCA} decomposition in terms of the mode-wise PCA residuals.

\begin{theorem}\label{thm: tensor error bound}
Let $\mathcal{X} \in \mathbb{R}^{J_{1} \times \cdots \times J_{N}}$ be an $N$-mode tensor, and let $\{U_n\}_{n=1}^{N}$ be the sparse projection matrices computed by Algorithm~\ref{alg:GeoSPCA-Tucker}. 
Then, the total reconstruction error $f(U_1, U_2, \ldots, U_N)$ in~\eqref{equ: fU1U2UN} is upper bounded by the sum of leading residual energies from classical PCA applied to the mode-$n$ unfoldings $X_{(n)}$ of $\mathcal{X}$:
\begin{equation}\label{equ: tensor error bound}
f(U_1, U_2, \ldots, U_N) \leq \sum_{n=1}^{N} \sum_{i \in \sigma^{n}} \left\| \epsilon_i^{n} \right\|^2,
\end{equation}
where $\epsilon^{n} = X_{(n)} - V^{n} {V^{n}}^{\top} X_{(n)}$ is the residual matrix from classical PCA, and $\sigma^{n} \subset [\prod_{i \ne n} J_i]$ contains the indices of the $k_n$ columns of $\epsilon^{n}$ with the highest norm.
\end{theorem}

\begin{proof}
See Appendix~\ref{app: appendix d}.
\end{proof}

\subsection{Complexity Analysis}\label{sec: Complexity Analysis}

The computational complexity of Algorithm~\ref{alg:GeoSPCA-Tucker} is primarily determined by its alternating optimization framework. In particular, the computational bottleneck lies in solving the subproblem formulated in problem~\eqref{equ:eta-constrained-problem}. The complexity of this subproblem can be analyzed in two parts: (i) the total number of generated cutting planes (cuts), and (ii) the computational cost per iteration.

In the worst case, the number of cuts can grow exponentially. Each iteration involves two main operations: (1) solving a binary linear optimization (BLO) problem, and (2) executing a separation oracle to detect violated constraints. The separation step requires solving a classical PCA problem over a matrix of dimension $J_n \times k_n$, which can be efficiently handled using singular value decomposition (SVD) with a computational cost of $O(k_n^3 + J_n k_n^2)$. According to the analysis in~\cite{bertsimas2022sparse}, the BLO problem can be solved using a tree search algorithm with a worst-case complexity of $O(k_n)$.
Therefore, the per-iteration complexity of solving Problem~\eqref{equ:eta-constrained-problem} is dominated by the cost of the SVD and is approximately $O(k_n^3 + J_n k_n^2)$. Although the total number of iterations could be exponential in theory, empirical evidence shows that the algorithm typically converges to high-quality solutions within a manageable number of iterations, even for moderately large problem sizes.
Summing across all tensor modes, the overall computational complexity of Algorithm~\ref{alg:GeoSPCA-Tucker} is $O(\sum_{n=1}^{N} (k_n^3 + J_n k_n^2))$.

Additionally, it is important to highlight that our method avoids the explicit computation of covariance matrices. This design choice not only reduces computational overhead but also significantly lowers memory consumption, particularly in high-dimensional scenarios where $J_n \ll \prod_{i \ne n} J_i$. As a result, problem~\eqref{equ:eta-constrained-problem} becomes especially attractive for large-scale tensor applications.

\section{EMPIRICAL RESULTS}\label{sec: EMPIRICAL RESULTS}

This section presents both synthetic and real-data experiments to evaluate the effectiveness of the proposed \textit{sparseGeoHOPCA} framework. 
We begin with support recovery in controlled simulations, then evaluate classification performance under high-dimensional compression. 
We further validate the method through an image reconstruction task on ImageNet, where \textit{sparseGeoHOPCA} outperforms matrix-based baselines in both visual quality and computational efficiency.

\textit{Computational resources.} All experiments were conducted on a workstation equipped with an Intel(R) Core(TM) i7-10700 CPU @ 2.90GHz, an NVIDIA RTX 4070 Super GPU, and 64GB RAM. The proposed sparseGeoHOPCA algorithm was implemented in Python using PyTorch 2.6 and Gurobi 10.0.1 for solving the binary linear optimization (BLO) subproblems.



\subsection{Synthetic Experiments}\label{sec: Synthetic Experiments}
We evaluate the support recovery performance of \textit{sparseGeoHOPCA} through controlled synthetic simulations based on a low-rank third-order tensor model, under varying sparsity and dimensionality conditions. 
The observed tensor is generated as
\begin{equation}\label{equ: X=d u v w}
    \mathcal{X} = \sum_{k=1}^{K} d_k\, \mathbf{u}_k \circ \mathbf{v}_k \circ \mathbf{w}_k + \mathcal{E}, \quad \mathcal{E}_{i,j,l} \overset{\text{iid}}{\sim} \mathcal{N}(0,1),
\end{equation}
where $K = 1$ and $d_1 = 100$. 
We consider four simulation scenarios with varying tensor dimensions and sparsity structures (details provided in Appendix~\ref{appendix:synthetic-setup}). 
The results demonstrate that our method is particularly effective in accurately identifying sparse structures, significantly outperforming baseline HOPCA in both precision and robustness.

Figure~\ref{fig:roc_comparison} shows ROC curves averaged over 50 replicates in Scenarios 1 and 2, where only mode-$u_1$ is sparse. In both settings, \textit{sparseGeoHOPCA} consistently achieves higher true positive rates and lower false positive rates compared to baseline methods, demonstrating robust support recovery under moderate sparsity and dimensional imbalance.

\begin{figure}[ht]
    \centering
    \includegraphics[width=\textwidth]{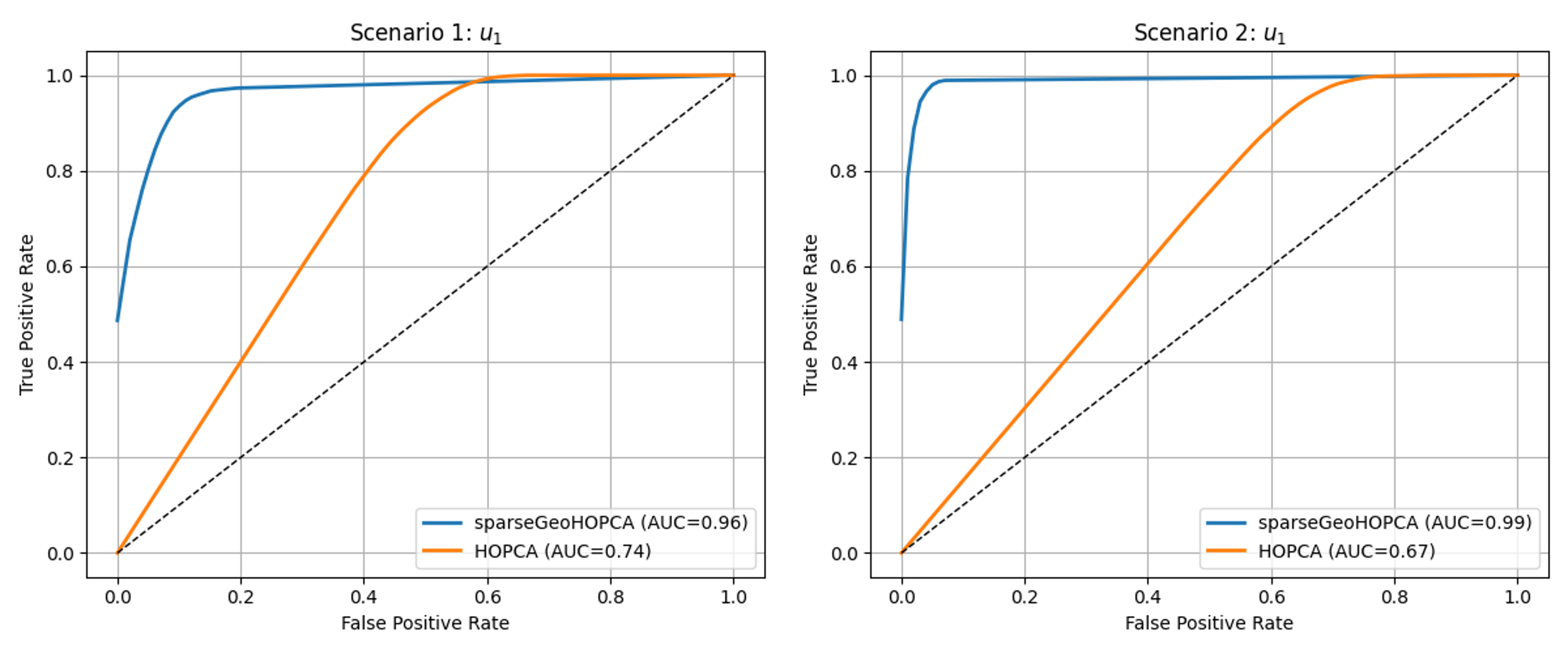}
    \caption{ROC curves for mode-$u_1$ in Scenarios 1 and 2, where it is the only sparse mode. Results are averaged over fifty replicates.}
    \label{fig:roc_comparison}
\end{figure}

Appendix~\ref{sec: ROC Analysis and Feature Selection Accuracy} includes ROC curves for Scenarios 3 and 4, and Appendix~\ref{appendix:tp_fp} reports true/false positive rate comparisons between HOPCA and \textit{sparseGeoHOPCA}.

\subsection{Classification with Compressed Features Extracted}
\label{sec: Classification with Compressed Features Extracted}

To assess the effectiveness of \textit{sparseGeoHOPCA} in preserving discriminative structures under compression, we conduct classification experiments on a reduced MNIST dataset. We apply the method to the sample mode of the training tensor and compress each class-specific subspace by retaining only a small number of representative directions. The test data remain uncompressed and are projected onto these compressed bases.

Table~\ref{tab:compression_accuracy} reports the overall classification accuracy under varying compression ratios. Even with a $10\times$ reduction in training dimensionality, the accuracy remains stable, demonstrating the robustness of the learned sparse representations.

\begin{table}[htbp]
\centering
\caption{Overall classification accuracy (\%) under varying compression ratios}
\label{tab:compression_accuracy}
\begin{tabular}{c|cccccc}
\toprule
Compression Ratio & 1.0 (no compression) & 0.8 & 0.6 & 0.4 & 0.2 & 0.1 \\
\midrule
Accuracy (\%)     & 87.75 & 86.25 & 86.50 & 86.88 & 86.50 & 84.62 \\
\bottomrule
\end{tabular}
\end{table}

Additional experimental details, including the classification framework, visualizations, and confusion matrices before and after compression, are provided in Appendix~\ref{appendix:classification}.

\subsection{Image Reconstruction}\label{sec: Image Reconstruction}
To further assess the applicability of \textit{sparseGeoHOPCA} to high-dimensional data, we conduct an image reconstruction experiment using selected samples from the ImageNet dataset \cite{russakovsky2015imagenet}. 
As most existing sparse tensor PCA methods lack public implementations and are designed for specific tensor structures, we compare our method with two representative matrix-based sparse PCA approaches.

We select two state-of-the-art sparse PCA methods designed for matrix data as baselines: sparsePCAChan~\cite{chan2015worst} and sparsePCABD~\cite{delefficient}. 
RGB images are first converted into matrix format by flattening the spatial dimensions. After extracting a fixed number of sparse components, we reconstruct the original image and compare the visual quality.

Figure~\ref{fig:imagenet_reconstruction} displays the reconstruction results. 
\textit{sparseGeoHOPCA} produces visibly sharper edges, fewer compression artifacts, and better texture preservation across all four image examples, despite being originally designed for tensors. 
It also achieves consistently lower reconstruction error and reduced runtime relative to the matrix-based baselines. 
These findings highlight both the accuracy and efficiency of \textit{sparseGeoHOPCA}, demonstrating its generalization capacity in structured dimensionality reduction tasks.

Full experimental details, including preprocessing steps, component settings, and descriptions of sparsePCAChan and sparsePCABD, as well as additional experimental results, are provided in Appendix~\ref{appendix:Details of Image Reconstruction Experiment}.

\begin{figure}[htbp]
    \centering
    \includegraphics[width=\textwidth]{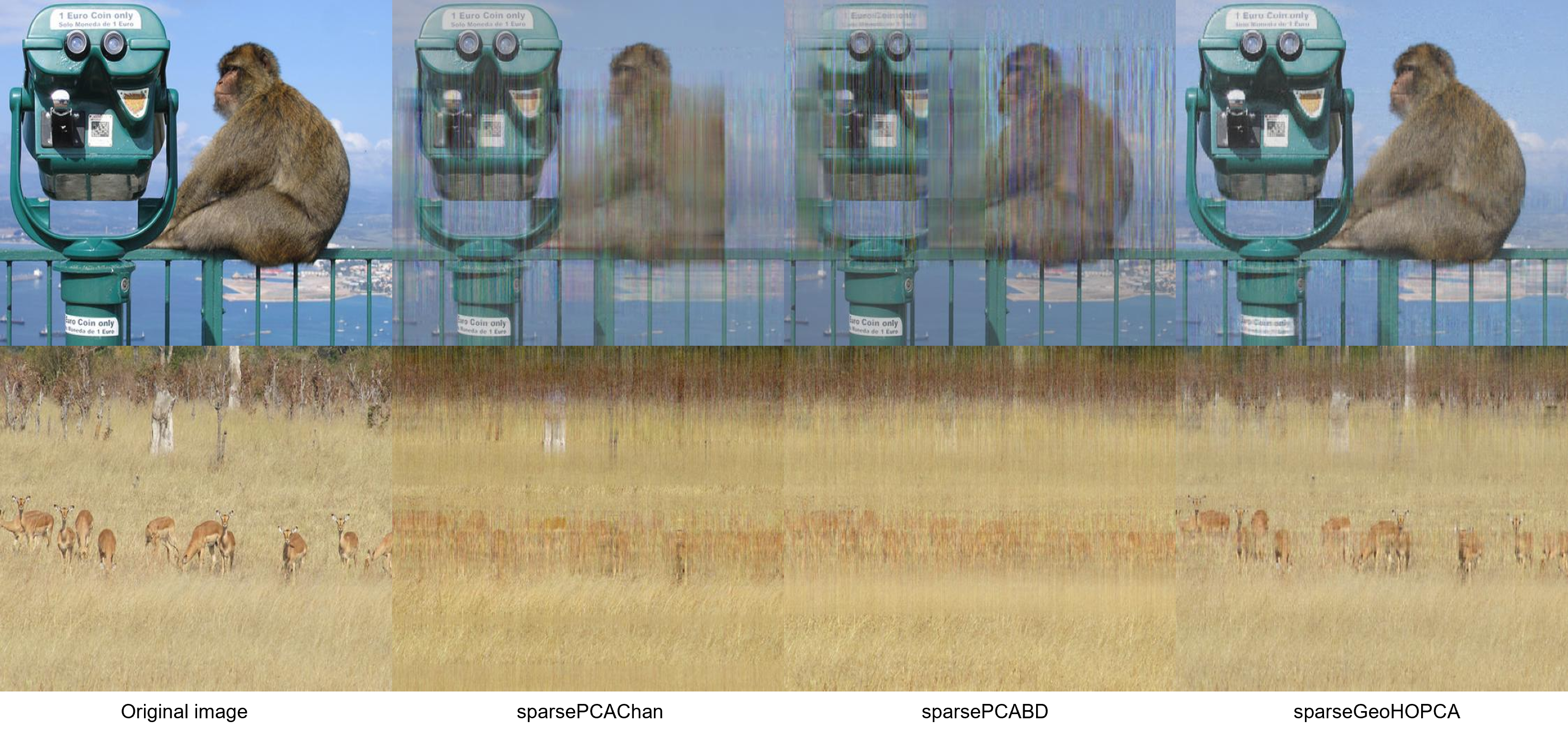}
    \caption{Visual comparison of image reconstruction results on ImageNet samples using three sparse PCA methods. From left to right: original image, sparsePCAChan, sparsePCABD, and our \textit{sparseGeoHOPCA}. In addition to preserving structural detail and reducing visual artifacts, our method also demonstrates lower reconstruction error and faster runtime compared to the matrix-based baselines.}
    \label{fig:imagenet_reconstruction}
\end{figure}

\section{Conclusion}\label{sec: conclusion}

We presented sparseGeoHOPCA, a geometry-aware framework for sparse higher-order principal component analysis (SHOPCA). By reformulating mode-wise sparse optimization as structured binary linear programs, our method eliminates the need for covariance estimation and deflation, enabling scalable and interpretable tensor decomposition.

Theoretical results establish the equivalence of the proposed subproblems to the original SHOPCA formulation and provide worst-case error bounds based on PCA residuals. The algorithm achieves linear computational complexity with respect to tensor size via alternating optimization.

Extensive experiments confirm that sparseGeoHOPCA achieves superior support recovery, robust classification under 10× compression, and effective image reconstruction, outperforming baseline methods while preserving structural discriminability in high-dimensional settings.

{
\newpage
\small
\bibliographystyle{abbrv}
\bibliography{reference, mybib}

}

\newpage
\appendix
\section{Proof of Theorem~\ref{thm: optimization decomposition}}\label{app: appendix a}
This appendix provides the detailed proof of Theorem~\ref{thm: optimization decomposition}.
Let $\{U_{i}\}_{i=1,\,i\ne n}^{N}$ denote the collection of all mode-$i$ projection matrices excluding mode $n$, i.e., $(U_1, \ldots, U_{n-1}, U_{n+1}, \ldots, U_N)$. We denote the objective function in \eqref{equ: f(U_{1},U_{2},,U_{N})} as $f(U_n)$ when all other $U_i$'s are fixed. Assume that each $U_i$ is column-orthonormal.

By properties of mode-wise tensor projections and the orthogonality of $U_i U_i^\top$, we have:
\begin{equation}
\begin{aligned}
f(U_n) 
&= \left\| \mathcal{X} - \mathcal{X} \times_1 U_1 U_1^\top \cdots \times_N U_N U_N^\top \right\|_F^2 \\
&= \left\| \mathcal{X} \times_N U_N U_N^\top - \mathcal{X} \times_1 U_1 U_1^\top \cdots \times_N U_N U_N^\top \right\|_F^2 \\
&\quad + \left\| \mathcal{X} - \mathcal{X} \times_N U_N U_N^\top \right\|_F^2 \\
&\quad \vdots \\
&= \left\| \mathcal{X} \times_{j \ne n} U_j U_j^\top - \mathcal{X} \times_{j \ne n} U_j U_j^\top \times_n U_n U_n^\top \right\|_F^2 \\
&\quad + \left\| \mathcal{X} - \mathcal{X} \times_{j \ne n} U_j U_j^\top \right\|_F^2,
\end{aligned}
\end{equation}

where $\times_{j \ne n}$ denotes the sequence of mode-$j$ projections over all $j \in \{1,\ldots,N\} \setminus \{n\}$. 
Noting that the second term in the final expression is independent of $U_n$, the optimization reduces to minimizing
$\left\| \mathcal{X} \times_{j \ne n} U_j^\top - \mathcal{X} \times_{j \ne n} U_j^\top \times_n U_n U_n^\top \right\|_F^2$.
This is equivalent to minimizing the matrix-form objective $\left\| X_{(n)} - U_n U_n^\top X_{(n)} \right\|_F^2$,
which corresponds precisely to problem~\eqref{equ: fU1U2UN} in the main text. This completes the proof.

\section{Proof of Theorem~\ref{thm: column optimaztion}}\label{app: appendix b}

This appendix provides the detailed proof of Theorem~\ref{thm: column optimaztion}. The proof is divided into two main steps.

\paragraph{Step 1: Reformulation of the Problem.}
We first show that the optimization problem~\eqref{equ: X_n-U_nU_n^TX_n} can be equivalently reformulated as the following problem:
\begin{equation}\label{equ:temp_equation}
\begin{aligned}
\underset{\mathbf{s} \in \{0,1\}^{\prod_{i \ne n} J_i},\; U_{n} \in \mathbb{R}^{J_{n} \times R_{n}}}{\text{maximize}} \quad & 
\sum_{j=1}^{\prod_{i \ne n} J_i} s_j \sum_{k=1}^{R_{n}} \left( X_{(n)}(:,j)^{\top} U_{n}(:,k) \right)^2 \\
\text{subject to} \quad & 
\mathbf{e}^\top \mathbf{s} \leq k_n, \quad 
U_{n}^{\top} U_{n} = I_{R_{n}}.
\end{aligned}
\end{equation}

\textit{Proof of Step 1.} \\

By the property of the Frobenius norm, problem~\eqref{equ: X_n-U_nU_n^TX_n} can be rewritten as:
\begin{equation}
    \begin{aligned}
         \underset{U_n}{\text{minimize}} \quad \left\| X_{(n)} - U_n U_n^\top X_{(n)} \right\|_F^2  
         &= \text{tr}\left( (X_{(n)} - U_n U_n^\top X_{(n)})^{\top}(X_{(n)} - U_n U_n^\top X_{(n)}) \right) \\
         &= \text{tr}(X_{(n)}^{\top}X_{(n)}) - \text{tr}(X_{(n)}^{\top}U_{n}U_{n}^{\top}X_{(n)}),
    \end{aligned}
\end{equation}
where the second equality follows from the orthogonality condition $U_{n}^{\top}U_{n}=I_{R_{n}}$.

Since the first term $\text{tr}(X_{(n)}^{\top}X_{(n)})$ is constant with respect to $U_n$, minimizing the objective is equivalent to maximizing
\begin{equation}\label{equ: left enigenvalue}
    \underset{U_n}{\text{maximize}} \quad \text{tr}(X_{(n)}^{\top}U_{n}U_{n}^{\top}X_{(n)}) = \underset{U_n}{\text{maximize}} \quad \text{tr}(U_{n}^{\top}X_{(n)}X_{(n)}^{\top}U_{n}).
\end{equation}

The optimization problem~\eqref{equ: left enigenvalue} is originally formulated as a left eigenvalue problem. 
We first transform it into an equivalent right eigenvalue formulation to facilitate the incorporation of sparsity constraints.

Specifically, we introduce an auxiliary binary vector $\mathbf{s} \in \{0,1\}^{\prod_{i \ne n} J_i}$, where $s_i = 0$ if the $i$-th row of $W$ is zero, and $s_i = 1$ otherwise. 
Define $S = \operatorname{diag}(\mathbf{s})$. 
The constraint $\mathbf{e}^\top \mathbf{s} \leq k_n$ ensures that the number of nonzero rows of $W$ does not exceed $k_n$.

Under this construction, the optimization problem becomes
\begin{equation}
    \begin{aligned}
\underset{\mathbf{s} \in \{0,1\}^{\prod_{i \ne n} J_i},\; W \in \mathbb{R}^{\prod_{i \ne n} J_i \times R_{n}}}{\text{maximize}} \quad & 
\text{tr}(W^{\top} S X_{(n)}^{\top} X_{(n)} S W) \\
\text{subject to} \quad & 
\mathbf{e}^\top \mathbf{s} \leq k_n, \quad 
W^{\top}W = I_{R_n}, \quad W^{\top}SW = I_{R_n}.
\end{aligned}
\end{equation}

The additional constraint $W^{\top} S W = I_{R_n}$ can be removed without loss of optimality, as it can be justified through singular value decomposition (SVD) analysis of $S X_{(n)}^{\top} = X_{(n)}(:,s)$.

We then convert the right eigenvalue formulation back to the original left eigenvalue form. 
By expanding the trace, the objective simplifies as:
\begin{equation}
    \begin{aligned}
        \text{tr}(U_{n}^{\top} X_{(n)} X_{(n)}^{\top} U_{n}) 
        &= \sum_{j=1}^{k_n} \sum_{k=1}^{R_n} \left( X_{(n)}(:,s_j)^{\top} U_n(:,k) \right)^2 \\
        &= \sum_{j=1}^{\prod_{i \ne n} J_i} s_j \sum_{k=1}^{R_n} \left( X_{(n)}(:,j)^{\top} U_n(:,k) \right)^2,
    \end{aligned}
\end{equation}
thus completing the proof that problem~\eqref{equ: X_n-U_nU_n^TX_n} is equivalent to problem~\eqref{equ:temp_equation}.

\paragraph{Step 2: Connection to the $\eta$-Constrained Selection Problem.}
Let $\mathbf{s}^0$ be an optimal solution to problem~\eqref{equ:temp_equation}. We now show that there exists a constant $\delta > 0$ such that for any $\eta \in [\eta(\mathbf{s}^0), \eta(\mathbf{s}^0) + \delta]$, any optimal solution to the $\eta$-constrained selection problem~\eqref{equ:eta-constrained-problem} is also an optimal solution to problem~\eqref{equ:temp_equation}.

\textit{Proof of Step 2.} \\

We proceed by contradiction. Suppose $\eta = \eta(\mathbf{s}^0)$, and assume that there exists a selection $\mathbf{s}' \in \{0,1\}^{\prod_{i \ne n} J_i}$ such that $\mathbf{s}'$ is an optimal solution to problem~\eqref{equ:eta-constrained-problem}, but not an optimal solution to problem~\eqref{equ:temp_equation}.

Then it follows that
\begin{equation}\label{equ:A}
\sum_{j=1}^{\prod_{i \ne n} J_i} s^0_j \sum_{k=1}^{R_n} \left( X_{(n)}(:,j)^{\top} U_n(:,k) \right)^2
>
\sum_{j=1}^{\prod_{i \ne n} J_i} s'_j \sum_{k=1}^{R_n} \left( X_{(n)}(:,j)^{\top} U_n(:,k) \right)^2.
\end{equation}

Since $\mathbf{s}'$ is feasible for problem~\eqref{equ:eta-constrained-problem}, we have
\begin{equation}\label{equ:B}
\sum_{j=1}^{\prod_{i \ne n} J_i} s'_j \left\| X_{(n)}(:,j) \right\|_F^2
\geq
\sum_{j=1}^{\prod_{i \ne n} J_i} s^0_j \left\| X_{(n)}(:,j) \right\|_F^2.
\end{equation}

Combining~\eqref{equ:A} and~\eqref{equ:B}, we conclude that
$$
\eta(\mathbf{s}') > \eta(\mathbf{s}^0) = \eta,
$$
which contradicts the feasibility condition $\eta(\mathbf{s}) \leq \eta$ required in problem~\eqref{equ:eta-constrained-problem}.

Since the feasible set $\{0,1\}^{\prod_{i \ne n} J_i}$ is finite and discrete, there exists a constant $\delta > 0$ such that the set of feasible solutions to problem~\eqref{equ:eta-constrained-problem} remains unchanged when $\eta$ varies within $[\eta(\mathbf{s}^0), \eta(\mathbf{s}^0) + \delta]$. 

Specifically, let $\hat{\mathbf{s}}$ be defined as
\begin{equation}
    \hat{\mathbf{s}} \in \arg\min_{\mathbf{s}}
\left\{
\begin{aligned}
& \left\| X_{(n)}(:,\mathbf{s}) - U_n[\mathbf{s}] U_n[\mathbf{s}]^\top X_{(n)}(:,\mathbf{s}) \right\|_F^2 \\
& \quad \text{subject to} \quad
\left\| X_{(n)}(:,\mathbf{s}) - U_n[\mathbf{s}] U_n[\mathbf{s}]^\top X_{(n)}(:,\mathbf{s}) \right\|_F^2 > \eta(\mathbf{s}^0)
\end{aligned}
\right\}.
\end{equation}

and define $\hat{\eta} = \left\| X_{(n)}(:,\hat{\mathbf{s}}) - U_n[\hat{\mathbf{s}}] U_n[\hat{\mathbf{s}}]^\top X_{(n)}(:,\hat{\mathbf{s}}) \right\|_F^2$. Then, we can set $\delta = \frac{\hat{\eta} - \eta(\mathbf{s}^0)}{2}$ to guarantee the stability of the optimal solution. This completes the proof.

\bigskip

Combining Step 1 and Step 2, this completes the proof of Theorem~\ref{thm: column optimaztion}.

\section{Proof of Theorem~\ref{thm: matrix error}}\label{app: appendix c}

This appendix provides the proof of Theorem~\ref{thm: matrix error}. The second inequality in~\eqref{equ: epsilon n bound} is straightforward. We now focus on proving the first inequality.

Consider
\begin{equation}
    \begin{aligned}
        \|S^{0}\epsilon^{n}\|_{F} 
        &= \|S^{0}(X_{(n)} - VV^{\top}X_{(n)})\|_{F} \\
        &= \|S^{0}X_{(n)} - S^{0}VV^{\top}X_{(n)}\|_{F} \\
        &= \|S^{0}X_{(n)} - VV^{\top}S^{0}X_{(n)}\|_{F} \\
        &\geq \|X_{(n)}(:,\mathbf{s}^{0}) - U_n[\mathbf{s}^{0}] U_n[\mathbf{s}^{0}]^{\top} X_{(n)}(:,\mathbf{s}^{0})\|_{F} \\
        &= \sqrt{\eta(\mathbf{s}^{0})},
    \end{aligned}
\end{equation}
where the third equality uses $S^{0}VV^{\top} = VV^{\top}S^{0}$, and the inequality follows from restricting the operation to the selected columns corresponding to $\mathbf{s}^0$.

This completes the proof of the boundary estimation~\eqref{equ: epsilon n bound}.

Furthermore, when $\eta = \|S^{0}\epsilon^{n}\|_{F}$, since $\|X_{(n)}(:,\mathbf{s}^{0}) - U_n[\mathbf{s}^{0}] U_n[\mathbf{s}^{0}]^{\top} X_{(n)}(:,\mathbf{s}^{0})\|_{F} \leq \|S^{0}\epsilon^{n}\|_{F}$, it follows that $\eta(\mathbf{s}^{0}) \leq \eta$, ensuring that $\mathbf{s}^{0}$ is feasible for problem~\eqref{equ:eta-constrained-problem}.

\section{Proof of Theorem~\ref{thm: tensor error bound}}\label{app: appendix d}

This appendix provides the proof of Theorem~\ref{thm: tensor error bound}.
We have
\begin{equation}
    \begin{aligned}
        f(U_1, U_2, \ldots, U_N) 
        &= \left\| \mathcal{X} - \mathcal{X} \times_1 U_1 U_1^\top \times_2 \cdots \times_N U_N U_N^\top \right\|_F^2 \\
        &\leq \sum_{n=1}^{N} \left\| \mathcal{X} - \mathcal{X} \times_n U_n U_n^\top \right\|_F^2 \\
        &= \sum_{n=1}^{N} \left\| X_{(n)} - U_n U_n^\top X_{(n)} \right\|_F^2 \\
        &\leq \sum_{n=1}^{N} \sum_{i \in \sigma^{n}} \left\| \epsilon_i^{n} \right\|^2,
    \end{aligned}
\end{equation}
where the first inequality follows from the results in~\cite{che2025efficientsiam,kolda2009tensor}, and the second inequality follows from Theorem~\ref{thm: matrix error}.

\section{Additional Details on Synthetic Experiments}\label{app: Additional Details on Synthetic Experiments}

\subsection{Simulation Setup}
\label{appendix:synthetic-setup}

We consider four simulation scenarios that vary in dimensionality and sparsity structure:
\begin{itemize}
    \item Scenario 1: $100 \times 100 \times 100$, sparsity in mode $\mathbf{U}$ only;
    \item Scenario 2: $1000 \times 20 \times 20$, sparsity in mode $\mathbf{U}$ only;
    \item Scenario 3: $100 \times 100 \times 100$, sparsity in all three modes;
    \item Scenario 4: $1000 \times 20 \times 20$, sparsity in all three modes.
\end{itemize}
For sparse modes, we randomly set 50\% of entries to zero, and the remaining entries are drawn from $N(0,1)$. 
For dense modes, the factors are obtained as the first $K$ left and right singular vectors of matrices with i.i.d. $N(0,1)$ entries.
\subsection{ROC Analysis and Feature Selection Accuracy}\label{sec: ROC Analysis and Feature Selection Accuracy}

To quantify the accuracy of support recovery, we report averaged Receiver Operating Characteristic (ROC) curves over 50 replications in each simulation setting.
Figure~\ref{fig:roc_comparison} displays ROC curves for mode-$u_1$ in Scenarios 1 and 2, where it is the only sparse mode. Figure~\ref{fig:roc_comparison2} presents ROC curves for modes $u_1$, $v_1$, and $w_1$ in Scenarios 3 and 4, where sparsity is present in all modes. In both cases, we compare the proposed \textit{sparseGeoHOPCA} method with a baseline HOPCA approach that applies naive thresholding to the components of Tucker decomposition~\cite{kolda2009tensor,kossaifi2019tensorly}.

As illustrated in Figures~\ref{fig:roc_comparison} and~\ref{fig:roc_comparison2}, \textit{sparseGeoHOPCA} consistently achieves higher true positive rates while maintaining substantially lower false positive rates across all simulation settings. This improvement is particularly evident in Scenarios 3 and 4, where the data exhibit full-mode sparsity and severe dimensional imbalance.
Moreover, the area under the ROC curve (AUC) highlights the robustness and reliability of sparseGeoHOPCA in high-dimensional, sparse tensor settings.

\begin{figure}[ht]
    \centering
    \includegraphics[width=\textwidth]{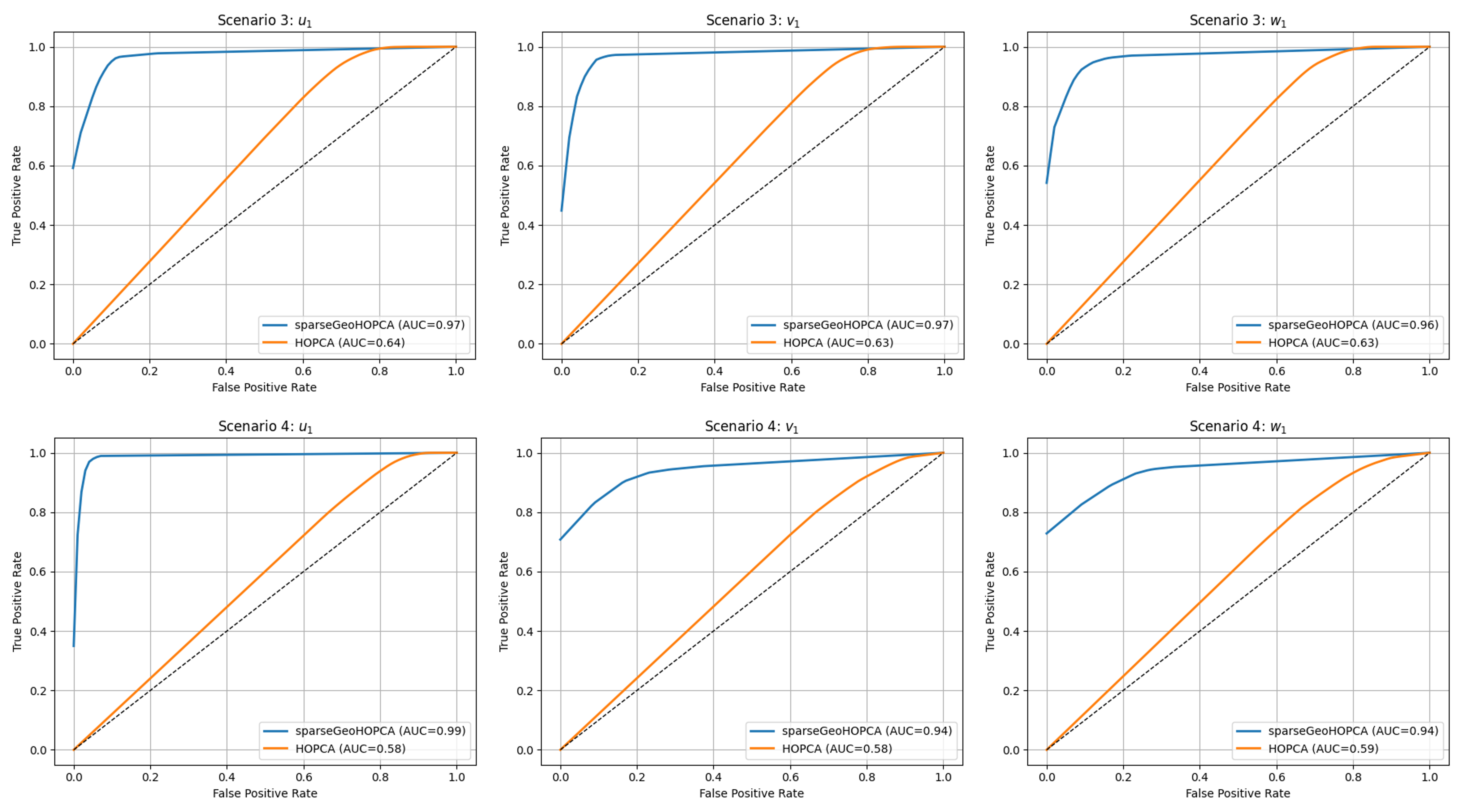}
    \caption{ROC curves for modes $u_1$, $v_1$, and $w_1$ in Scenarios 3 and 4, where sparsity is present in all modes. Results are averaged over fifty independent runs.}
    \label{fig:roc_comparison2}
\end{figure}

\subsection{True and False Positive Rate Comparison}
\label{appendix:tp_fp}
Table~\ref{tab:tp_fp_refined} summarizes the mean and standard deviation of true positive (TP) and false positive (FP) rates for sparseGeoHOPCA and HOPCA across all modes and scenarios. 
The results reinforce the ROC analysis by demonstrating that sparseGeoHOPCA not only yields high TP rates but also significantly reduces FP rates compared to HOPCA.

\begin{table}[htbp]
\centering
\caption{TP/FP comparison between HOPCA and sparseGeoHOPCA across scenarios}
\begin{tabular}{|c|c|c|c|c|}
\hline
\textbf{Scenario} & \textbf{Mode} & \textbf{Method} & \textbf{TP (mean ± std)} & \textbf{FP (mean ± std)} \\
\hline
\multirow{2}{*}{1} 
& $u_1$ & sparseGeoHOPCA & 0.967 ± 0.048 & 0.041 ± 0.052 \\
&       & HOPCA          & 1.000 ± 0.000 & 0.514 ± 0.082 \\
\hline
\multirow{2}{*}{2} 
& $u_1$ & sparseGeoHOPCA & 0.988 ± 0.015 & 0.012 ± 0.017 \\
&       & HOPCA          & 1.000 ± 0.000 & 0.670 ± 0.077 \\
\hline
\multirow{6}{*}{3} 
& $u_1$ & sparseGeoHOPCA & 0.972 ± 0.040 & 0.034 ± 0.051 \\
&       & HOPCA          & 1.000 ± 0.000 & 0.730 ± 0.076 \\
\cline{2-5}
& $v_1$ & sparseGeoHOPCA & 0.968 ± 0.048 & 0.031 ± 0.038 \\
&       & HOPCA          & 1.000 ± 0.000 & 0.745 ± 0.071 \\
\cline{2-5}
& $w_1$ & sparseGeoHOPCA & 0.962 ± 0.054 & 0.034 ± 0.053 \\
&       & HOPCA          & 1.000 ± 0.000 & 0.733 ± 0.076 \\
\hline
\multirow{6}{*}{4} 
& $u_1$ & sparseGeoHOPCA & 0.989 ± 0.016 & 0.014 ± 0.017 \\
&       & HOPCA          & 1.000 ± 0.000 & 0.839 ± 0.072 \\
\cline{2-5}
& $v_1$ & sparseGeoHOPCA & 0.928 ± 0.098 & 0.044 ± 0.091 \\
&       & HOPCA          & 1.000 ± 0.000 & 0.845 ± 0.112 \\
\cline{2-5}
& $w_1$ & sparseGeoHOPCA & 0.929 ± 0.096 & 0.045 ± 0.095 \\
&       & HOPCA          & 1.000 ± 0.000 & 0.826 ± 0.120 \\
\hline
\end{tabular}
\label{tab:tp_fp_refined}
\end{table}

While HOPCA achieves perfect TP rates in all cases, it suffers from excessive false positives, often exceeding 70\% in more challenging configurations, indicating poor feature specificity. 
In contrast, sparseGeoHOPCA delivers a more balanced and controlled feature selection. 
Notably, in Scenario 4 mode $v_1$, sparseGeoHOPCA achieves a TP rate of $0.928 \pm 0.098$ and an FP rate of $0.044 \pm 0.091$, whereas HOPCA yields an FP rate as high as $0.845 \pm 0.112$.

These findings demonstrate that sparseGeoHOPCA provides a more effective solution for sparse tensor decomposition when accurate support recovery is critical, especially under limited sample sizes and high ambient dimensionality.

\section{Classification with Compressed Features Extracted}\label{appendix:classification}

The proposed sparseGeoHOPCA algorithm is designed to extract informative and sparse representations from high-dimensional tensor data while preserving structural integrity across all modes. It enables simultaneous multimodal coclustering and feature selection. To validate the effectiveness of the method in real-world classification tasks, we apply it to handwritten digit recognition using the MNIST dataset \cite{lecun1998mnist}. MNIST is a benchmark dataset comprising 60,000 training and 10,000 testing images, where each image is a $28 \times 28$ grayscale representation of a digit (0--9).

The fundamental assumptions underlying image classification are: (1) samples from the same class share common latent features, and (2) samples from different classes exhibit distinct structural patterns. Therefore, a successful classification algorithm should be able to extract discriminative features from raw data. In this study, we focus on supervised classification, where class labels are known during training.

We adopt a projection-based classification framework, following the methodology proposed in \cite{keegan2021tensor,newman2017image}. In the training phase, a local feature subspace (or “local basis”) is constructed for each class using the respective training samples. This subspace captures intra-class variation. During testing, each test image $b$ is orthogonally projected onto all class-specific subspaces, and the classification decision is made by selecting the subspace that yields the smallest Frobenius norm between the projection and the original sample.
\begin{figure}[htbp]
    \centering
    \begin{subfigure}[t]{0.48\textwidth}
        \centering
        \includegraphics[width=\textwidth]{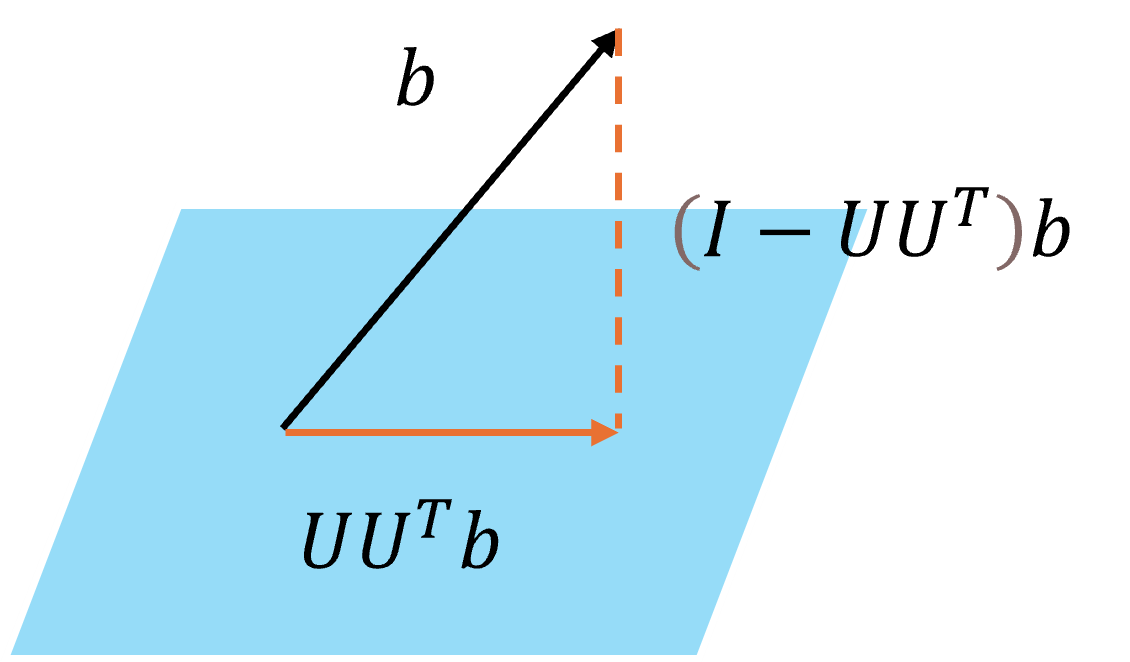}
        \caption{Projection of a test image $b$ onto a single subspace $UU^T$. The residual is $(I - UU^T)b$.}
    \end{subfigure}
    \hfill
    \begin{subfigure}[t]{0.48\textwidth}
        \centering
        \includegraphics[width=\textwidth]{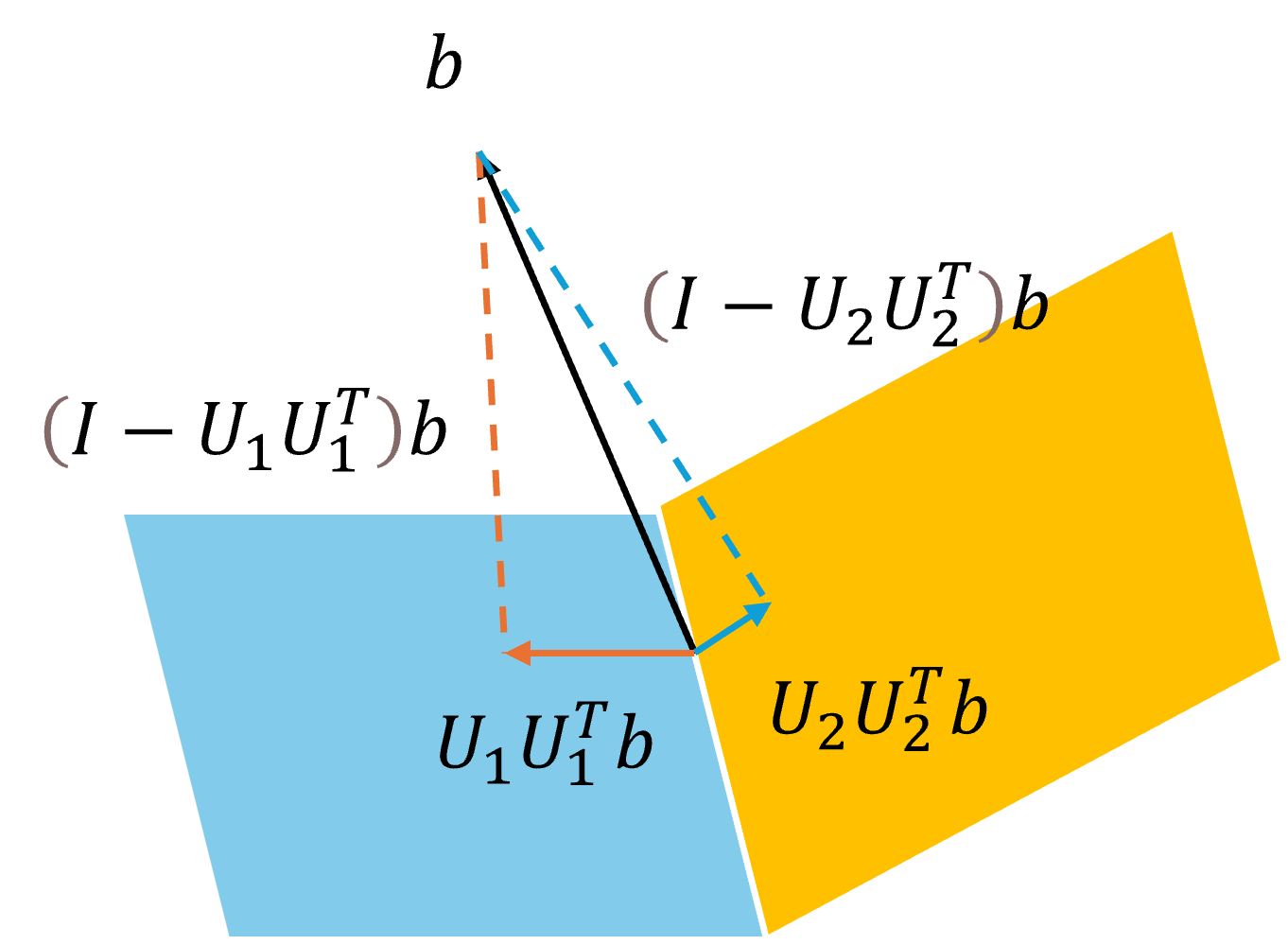}
        \caption{Projection of $b$ onto multiple class-specific subspaces $U_1U_1^T$, $U_2U_2^T$, etc.}
    \end{subfigure}
    \caption{Illustration of the projection-based classification scheme. Classification is based on the subspace yielding the smallest projection residual.}
    \label{fig:proj12}
\end{figure}
Figure~\ref{fig:proj12} illustrates this framework. In the left panel, a test image $b$ is projected onto a single class-specific subspace $UU^T$, with residual $(I - UU^T)b$. In practice, the image is projected onto multiple subspaces $\{U_i U_i^T\}$, and the label is predicted based on the minimum projection error, as shown in the right panel.

The success of this approach critically depends on constructing representative subspaces for each class. To this end, we apply sparseGeoHOPCA independently to the training samples of each class. This process yields sparse, interpretable basis components that enhance both intra-class consistency and inter-class separability.

To visually demonstrate the effectiveness of the extracted features, we use digits 7 and 8 from MNIST as an illustrative case. Figure~\ref{fig:example78} presents representative training images for both classes. Although intra-class variability exists, structural differences between classes are visually apparent.

\begin{figure}[htbp]
    \centering
    \begin{subfigure}[t]{0.48\textwidth}
        \centering
        \includegraphics[width=\textwidth]{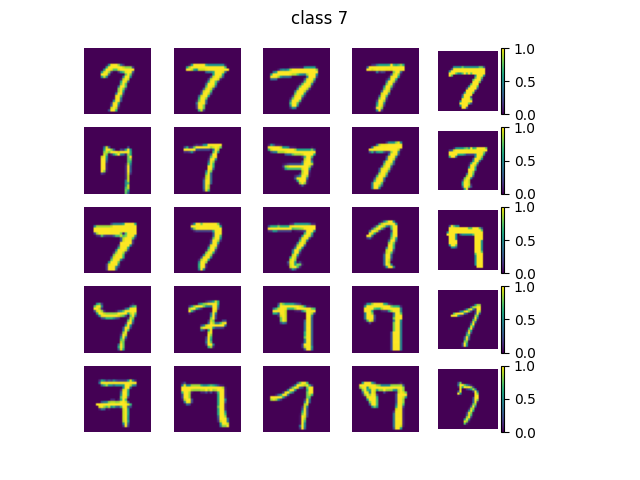}
        \caption{Sample training images from digit class 7.}
    \end{subfigure}
    \hfill
    \begin{subfigure}[t]{0.48\textwidth}
        \centering
        \includegraphics[width=\textwidth]{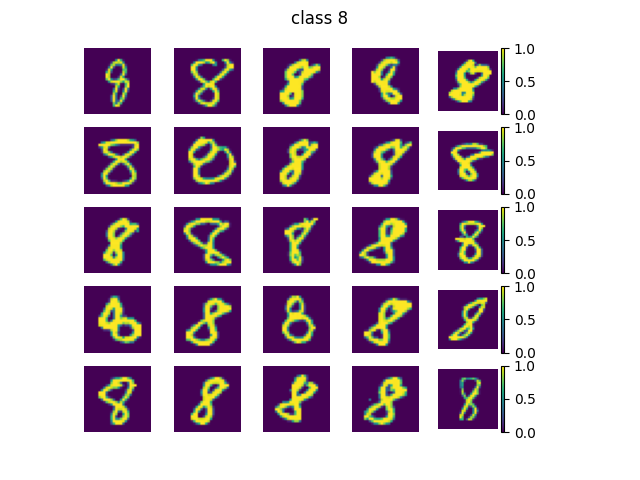}
        \caption{Sample training images from digit class 8.}
    \end{subfigure}
    \caption{Training samples used to construct class-specific subspaces.}
    \label{fig:example78}
\end{figure}

We then apply sparseGeoHOPCA to the training tensors of digits 7 and 8 and visualize the first two basis vectors extracted along the sample mode (i.e., the mode corresponding to training indices). As shown in Figure~\ref{fig:Uvis}, the learned components capture digit-specific structures: class 7 shows strong vertical and angular strokes, while class 8 reveals looped patterns. These localized and interpretable features improve downstream classification performance.

\begin{figure}[htbp]
    \centering
    \begin{subfigure}[t]{0.48\textwidth}
        \centering
        \includegraphics[width=\textwidth]{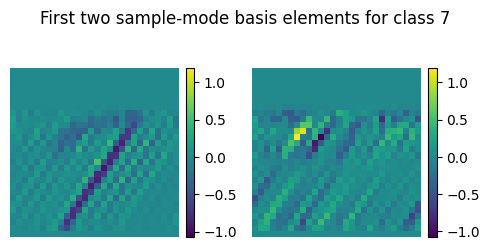}
        \caption{First two basis elements from the sample mode for class 7.}
    \end{subfigure}
    \hfill
    \begin{subfigure}[t]{0.48\textwidth}
        \centering
        \includegraphics[width=\textwidth]{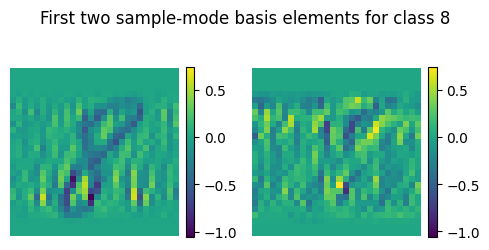}
        \caption{First two basis elements from the sample mode for class 8.}
    \end{subfigure}
    \caption{Visualization of class-specific basis vectors extracted by sparseGeoHOPCA.}
    \label{fig:Uvis}
\end{figure}

\paragraph{Evaluation of Classification Performance Before and After Compression}

To evaluate the effectiveness of sparseGeoHOPCA in preserving discriminative structures under compression, we conduct experiments on a reduced version of MNIST. Specifically, we use 500 training and 80 test samples per class, resulting in a total of 5,000 training and 800 test samples.

To simulate data compression, we apply sparseGeoHOPCA to the sample mode of the training tensor and reduce its dimensionality by a factor of ten. That is, for each class, only 50 representative directions are retained after decomposition. The test set remains uncompressed and is projected onto the compressed class-specific bases.

Figure~\ref{fig:confmat_compare} displays the normalized confusion matrices obtained before and after compression. The baseline case (left panel) achieves an overall accuracy of 87.75\%, while the compressed case (right panel) achieves 84.62\%. Despite the tenfold reduction in training dimension, the classification structure remains largely intact. In particular, digits such as 0, 1, and 4 retain high precision, while digits like 3, 5, and 8 are more sensitive to the compression due to higher intra-class variability.

To further analyze classification robustness under different compression levels, we evaluate the overall accuracy for varying compression ratios. As shown in Table~\ref{tab:compression_accuracy}, the model maintains relatively stable accuracy even when only 10\% of the training data (in the sample mode) is retained.

\begin{figure}[htbp]
    \centering
    \begin{subfigure}[t]{0.48\textwidth}
        \centering
        \includegraphics[width=\textwidth]{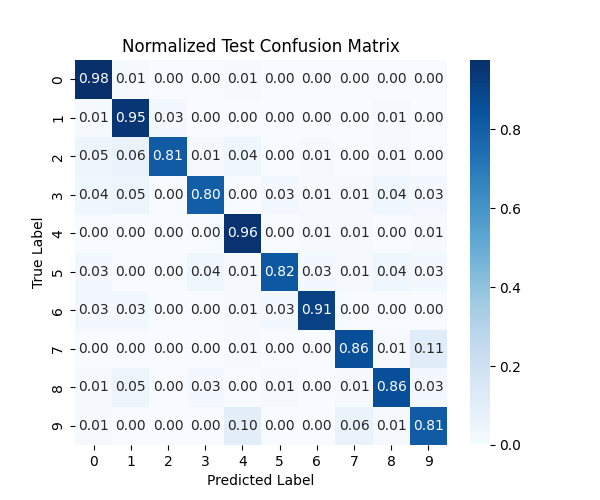}
        \caption{Normalized confusion matrix using original data (Accuracy: 87.75\%).}
    \end{subfigure}
    \hfill
    \begin{subfigure}[t]{0.48\textwidth}
        \centering
        \includegraphics[width=\textwidth]{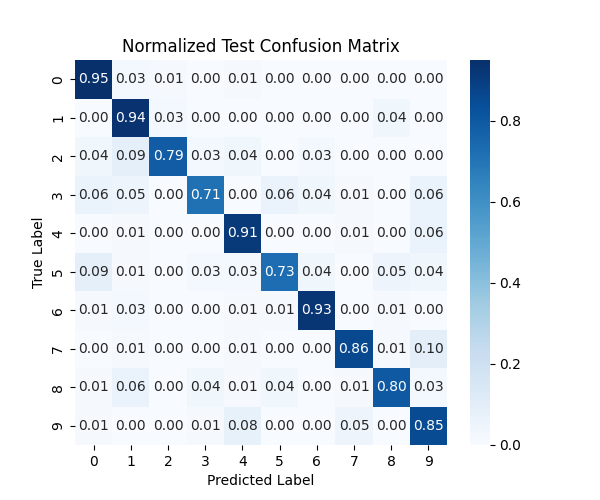}
        \caption{Normalized confusion matrix after 10$\times$ compression (Accuracy: 84.62\%).}
    \end{subfigure}
    \caption{Comparison of classification performance before and after sparseGeoHOPCA-based sample-mode compression. Each class uses 500 training and 80 test samples.}
    \label{fig:confmat_compare}
\end{figure}

\section{Details of Image Reconstruction Experiment}
\label{appendix:Details of Image Reconstruction Experiment}

\paragraph{Dataset and Preprocessing.}
We randomly select four RGB images from the ImageNet dataset~\cite{russakovsky2015imagenet}, with original resolutions of $500 \times 368 \times 3$, $500 \times 375 \times 3$, $500 \times 375 \times 3$, and $500 \times 359 \times 3$, respectively. To facilitate matrix-based analysis, each image $\mathbf{I} \in \mathbb{R}^{m \times n \times 3}$ is reshaped in two modes:
(i) as a row-wise matrix $\mathbf{I}_r \in \mathbb{R}^{m \times 3n}$ by flattening the RGB channels along columns;
(ii) as a column-wise matrix $\mathbf{I}_c \in \mathbb{R}^{n \times 3m}$ by flattening the RGB channels along rows.
These representations allow structured feature extraction along spatial or chromatic dimensions.

\paragraph{Baselines.}
We compare our method against two representative state-of-the-art baselines:
\textbf{(1) Chan’s algorithm (sparsePCAChan):} a polynomial-time approximation algorithm with provable multiplicative guarantees for sparse PCA~\cite{chan2015worst}. 
\textbf{(2) Block-diagonalization-based method (sparsePCABD):} a heuristic approach that transforms the input matrix into a block-diagonal form to facilitate sparse component extraction~\cite{delefficient}.

All methods retain an equal number of principal components (90), and reconstruction is performed via linear combination of the selected bases. 
Visual results are shown in Figure~\ref{fig:imagenet_reconstruction} in the main paper and Figure~\ref{fig:imagenet_reconstruction2}.

\paragraph{Observations.}
Compared to the matrix-specific methods, \textit{sparseGeoHOPCA} yields sharper reconstructions with fewer vertical or horizontal artifacts. 
This indicates that the geometry-aware support selection mechanism, originally designed for tensors, transfers well to structured matrices, especially when preserving global structure is critical.

\begin{figure}[htbp]
    \centering
    \includegraphics[width=\textwidth]{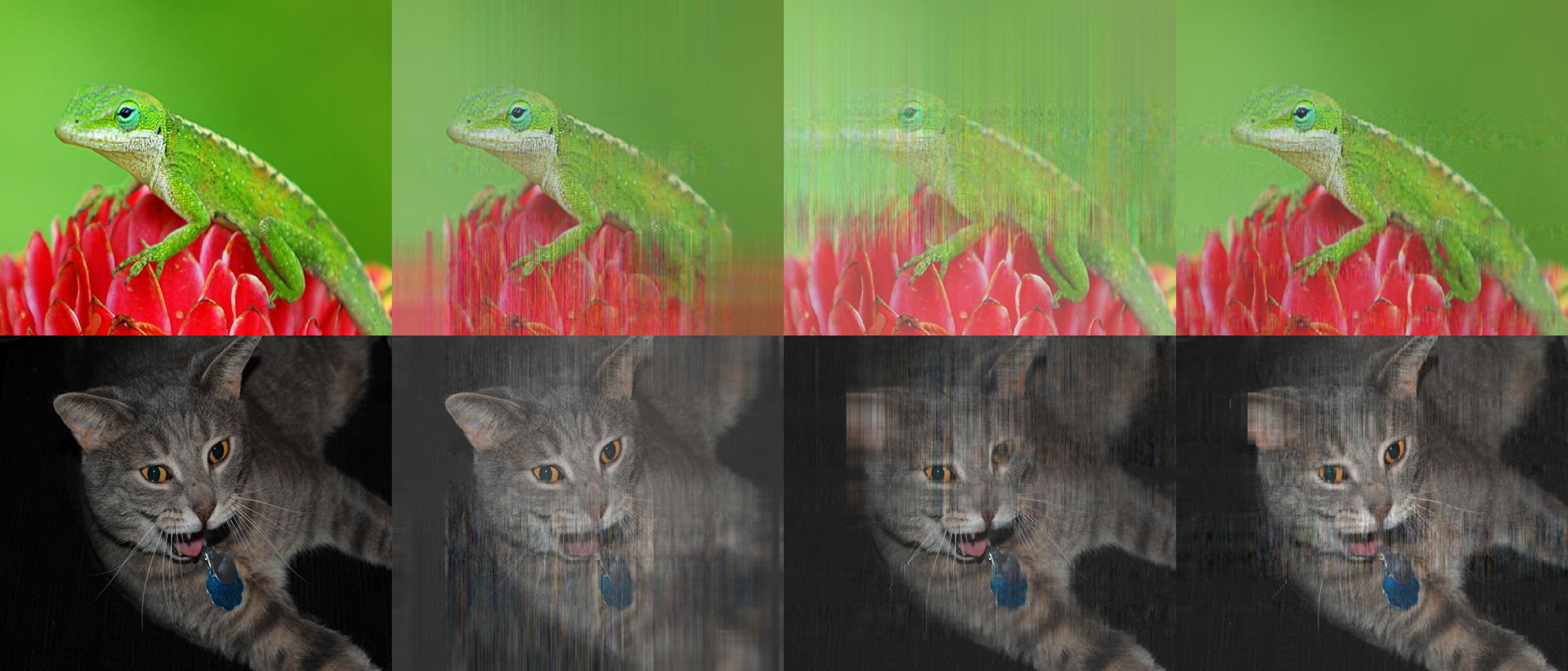}
    \caption{Visual comparison of image reconstruction results on ImageNet samples using three sparse PCA methods. From left to right: original image, sparsePCAChan, sparsePCABD, and our \textit{sparseGeoHOPCA}. In addition to preserving structural detail and reducing visual artifacts, our method also demonstrates lower reconstruction error and faster runtime compared to the matrix-based baselines.}
    \label{fig:imagenet_reconstruction2}
\end{figure}

Quantitative metrics (e.g., PSNR or MSE) may be added in future work to further support visual findings. Compared to matrix-specific methods, \textit{sparseGeoHOPCA} achieves sharper visual reconstructions with significantly fewer directional artifacts (e.g., vertical or horizontal banding).
Beyond qualitative improvements, our method also yields lower reconstruction error (measured by Frobenius norm) and faster runtime, demonstrating both numerical accuracy and computational efficiency. As summarized in Table~\ref{tab:recon_error_time}, \textit{sparseGeoHOPCA} consistently achieves the lowest reconstruction error across all four test images, with Frobenius norms significantly lower than those of matrix-based baselines. Furthermore, it matches or outperforms the baselines in runtime, achieving the fastest execution in three out of four cases. These results highlight the effectiveness of geometry-aware support selection, even when applied to flattened matrix data.

\begin{table}[htbp]
\centering
\small
\caption{Reconstruction error (Frobenius norm) and runtime (in seconds) for each method across four RGB images.}
\label{tab:recon_error_time}
\begin{tabular}{c|cc|cc|cc|cc}
\toprule
\multirow{2}{*}{Method} & \multicolumn{2}{c|}{Image 1 in Figure~\ref{fig:imagenet_reconstruction}} & \multicolumn{2}{c|}{Image 2 in Figure~\ref{fig:imagenet_reconstruction}} & \multicolumn{2}{c|}{Image 3 in Figure~\ref{fig:imagenet_reconstruction2}} & \multicolumn{2}{c}{Image 4 in Figure~\ref{fig:imagenet_reconstruction2}} \\
& Error & Time & Error & Time & Error & Time & Error & Time \\
\midrule
sparsePCAChan    & 70.3 & \textbf{0.8} & 44.4 & 0.9 & 67.6 & \textbf{0.8} & 39.5 & 0.8 \\
sparsePCABD      & 89.5 & 2.3 & 44.2 & 2.5 & 53.6 & 2.4 & 32.4 & 2.1 \\
sparseGeoHOPCA   & \textbf{32.3} & \textbf{0.8} & \textbf{35.5} & \textbf{0.8} & \textbf{33.5} & \textbf{0.8} & \textbf{30.9} & \textbf{0.6} \\

\bottomrule
\end{tabular}
\end{table}

These results suggest that the geometry-aware support selection mechanism, originally designed for tensors—transfers effectively to structured matrices, particularly when global structural fidelity is essential. Quantitative image quality metrics such as PSNR or SSIM may be incorporated in future work to complement these findings.

\section{Limitations}\label{app: Limitations}
While the proposed framework demonstrates strong empirical performance across diverse settings, we note that the final decomposition may exhibit mild sensitivity to the initial support set selection, particularly in extremely noisy scenarios. However, such sensitivity does not materially impact the overall support recovery or classification accuracy, as confirmed by extensive replicates and robustness checks.

\section{Broader Impacts}\label{sec: Boroader Impacts}
This paper contributes a general-purpose method for sparse tensor decomposition, with potential benefits in high-dimensional data analysis across scientific, engineering, and biomedical domains. By improving the interpretability and computational efficiency of higher-order PCA, the proposed framework may facilitate downstream applications such as multimodal learning, neuroimaging, or large-scale sensor data processing. These positive impacts are indirect and contingent upon domain-specific adoption. 

We do not anticipate any direct negative societal impact from this work. The method does not involve sensitive data, human subjects, or decision-making in high-stakes settings. Moreover, the algorithm itself is neutral and intended for general-purpose use in scientific computing. Nevertheless, as with any representation learning technique, care should be taken when applying the method in sensitive contexts (e.g., medical diagnostics or surveillance), where interpretability and fairness are critical.

As a theoretical and algorithmic contribution, the work primarily advances mathematical tools for structured dimensionality reduction, rather than deploying an application-specific pipeline.

\end{document}